\documentclass[12pt]{article}
\usepackage{graphicx} % Required for inserting images

\usepackage[all]{xy}
\usepackage{tikz-cd}
\usepackage{amsmath,amssymb,xcolor,enumerate}
\makeatletter
\let\old@makechapterhead\@makechapterhead
\def\fake@makechapterhead#1{%
	\vspace*{50\p@}%
	{\parindent \z@ \raggedright \normalfont
		\ifnum \c@secnumdepth >\m@ne
		\huge\bfseries \strut%\@chapapp\space \thechapter
		\par\nobreak
		\vskip 20\p@
		\fi
		\interlinepenalty\@M
		\Huge \bfseries #1\par\nobreak
		\vskip 40\p@
}}
\newcommand{\newchapterhead}{\let\@makechapterhead\fake@makechapterhead}
\newcommand{\restorechapterhead}{\let\@makechapterhead\old@makechapterhead}
\makeatother

\usepackage{dsfont}
\usepackage{tikz}
\usetikzlibrary{calc}
\usepackage {amsthm}
\usepackage {amssymb}
\usepackage {latexsym}
\usepackage {amsfonts}
\usepackage {mathrsfs}
\usepackage {epsfig}
\usepackage {color}
\usepackage {graphicx}
\usepackage{bbold}
\usepackage{xfrac}
\usepackage{tocloft}
\makeatletter
\@addtoreset{section}{part}
\makeatother
\newlength\mylen

\renewcommand\cftpartpresnum{Part~}
\usepackage[margin=1.1in]{geometry}
\newtheorem{theorem}{Theorem}[section]

\usepackage[final]{pdfpages}
\usepackage{mathtools}
\usepackage{amsmath}
\usepackage{amsfonts}
\usepackage{amssymb}
\usepackage{amsthm}
\usepackage{graphicx}
\usepackage{dsfont}
\usepackage{hyperref}
\usepackage{mathabx}
\usepackage{tikz}
\usepackage{comment}
\usepackage{float}
\usepackage{lineno}

\newcommand{\red}[1]{{\color{red}#1}}

\usepackage{url}
\makeatletter
\g@addto@macro{\UrlBreaks}{\UrlOrds}
\makeatother
\usepackage{multicol}
\usepackage{lineno}
\usepackage{cleveref}
\crefformat{section}{\S#2#1#3}

\theoremstyle{definition}

\newtheorem{thm}[theorem]{Theorem}
\newtheorem{lem}[theorem]{Lemma}
\newtheorem{cor}[theorem]{Corollary}
\newtheorem{obs}[theorem]{Observation}

\allowdisplaybreaks

\newtheoremstyle{named}{}{}{\itshape}{}{\bfseries}{.}{.5em}{\thmnote{#3's }#1}
\theoremstyle{named}

\newcommand{\bfDelta}{\boldsymbol{\Delta}}

\renewcommand{\epsilon}{\varepsilon}

\numberwithin{equation}{section}

\title{Reals in the Matet and Willow models}
\author{Raiean Banerjee}
\date{}

\begin{document}

\maketitle

\section{Introduction}
We visit Brendle's \emph{Uniform Forcings Diagram} from \cite[page 6]{brendle1995}:
\begin{equation*}
\xymatrix@R=1mm@C=1mm{
\mathbb{S} & \supseteq & \mathbb{M} & \supseteq & \mathbb{L}\\
\rotatebox[origin=c]{270}{$\supseteq$} & &
\rotatebox[origin=c]{270}{$\supseteq$} & &
\\
\mathbb{W} & \supseteq & \mathbb{T} &  & \rotatebox[origin=c]{270}{$\supseteq$}\\
\rotatebox[origin=c]{270}{$\supseteq$} & & & \rotatebox[origin=c]{315}{$\supseteq$}\\
\mathbb{V} & & \supseteq & &  \mathbb{R}
}\end{equation*}
The inclusions in this diagram immediately give rise to implications between the corresponding
regularity properties
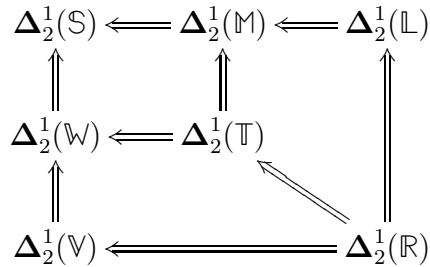
\begin{equation*}
\xymatrix@R=3mm@C=3mm{
\bfDelta^1_2(\mathbb{S})\ar@{<=}[rr] &  & \bfDelta^1_2(\mathbb{M})\ar@{<=}[rr] &  & \bfDelta^1_2(\mathbb{L})\\
 & & & &
\\
\bfDelta^1_2(\mathbb{W})\ar@{=>}[uu] & & \bfDelta^1_2(\mathbb{T})\ar@{=>}[ll]\ar@{=>}[uu] &  & \\
 & & & \\
\bfDelta^1_2(\mathbb{V})\ar@{=>}[uu] & &  & &  \bfDelta^1_2(\mathbb{R})\ar@{=>}[uuuu]\ar@{=>}[uull]
\ar@{=>}[llll]
}\end{equation*}
which we shall call the \emph{Uniform Regularities Diagram}.
We believe that the Uniform Regularities Diagram is complete in the sense that the implications marked in the diagram are the only ones that exist. The status of the subdiagram with Matet and Willowtree forcing removed was known, i.e., that the diagram 
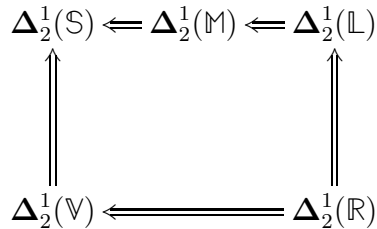
\begin{equation*}
\xymatrix@R=8mm@C=5mm{
\bfDelta^1_2(\mathbb{S})\ar@{<=}[r]  & \bfDelta^1_2(\mathbb{M})\ar@{<=}[r]  & \bfDelta^1_2(\mathbb{L})\\
 & &  \\
\bfDelta^1_2(\mathbb{V})\ar@{=>}[uu] & & \bfDelta^1_2(\mathbb{R})\ar@{=>}[uu]
\ar@{=>}[ll]
}\end{equation*}
is complete in the above sense. We should like to emphasise that 
one component of this completeness is 
the fact that 
$\bfDelta^1_2(\mathbb{L})$ does not imply 
$\bfDelta^1_2(\mathbb{V})$ which was proved in \red{cite the paper}.

\begin{obs}\label{obs:diagramcompleteness}
The Uniform Regularities Diagram will be complete if the following non-implications hold:
\begin{enumerate}[(a)]\setlength{\itemsep}{0mm}
  \item $\bfDelta^1_2(\mathbb{T})\not\Rightarrow\bfDelta^1_2(\mathbb{V})$,
  \item $\bfDelta^1_2(\mathbb{T})\not\Rightarrow\bfDelta^1_2(\mathbb{L})$, and
  \item $\bfDelta^1_2(\mathbb{L})\not\Rightarrow\bfDelta^1_2(\mathbb{W})$.
  \end{enumerate}

\end{obs}

\begin{proof}
We go through all non-implications that need to be checked. 

By the completeness of the subdiagram with Matet and Willowtree forcing removed, we only need to show the non-implications with the two additional forcings for the forcings 
$\mathbb{S}$, 
$\mathbb{M}$,
$\mathbb{L}$, and
$\mathbb{V}$. (The forcing 
$\mathbb{R}$ does not have any non-implications in the Uniform Regularities Diagram.)
Sacks, Miller and Laver regularity cannot imply either Matet or Willowtree regularity
by (c) and transitivity. Since Silver forcing does not 
does not add unbounded reals (cf.\ \cite[Proposition 4.2]{BHL05}), we have
$\bfDelta^1_2(\mathbb{V})\not\Rightarrow\bfDelta^1_2(\mathbb{M})$; thus Silver regularity cannot
imply Matet regularity by transitivity.

Again, since $\bfDelta^1_2(\mathbb{V})\not\Rightarrow\bfDelta^1_2(\mathbb{M})$, by transitivity, Willowtree regularity cannot imply Miller regularity
(and therefore not Matet, Laver, or Mathias regularity).
Finally, the Matet non-implications all follow directly from (a) and (b).\end{proof}

In this paper, we shall 
prove two of the assumptions of Observation \ref{obs:diagramcompleteness}:
statement (a) in Corollary \ref{cor:matetvssilver} and a weaker version of (c), viz.
$\bfDelta^1_2(\mathbb{S})\not\Rightarrow\bfDelta^1_2(\mathbb{W})$
(cf.\ Corollary \ref{cor:sacksnotwillow}). The combination of Corollaries 
\ref{cor:matetvssilver} \& \ref{cor:sacksnotwillow} implies that the following subdiagram is complete:
$$
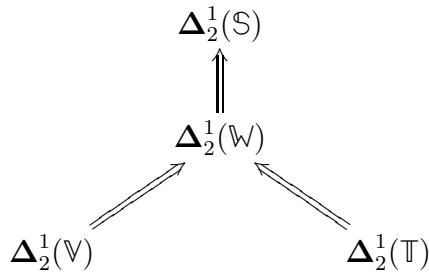
\xymatrix{
 & \bfDelta^1_2(\mathbb{S}) & \\
 & \bfDelta^1_2(\mathbb{W})\ar@{=>}[u] & \\
\bfDelta^1_2(\mathbb{V})\ar@{=>}[ur] & & \bfDelta^1_2(\mathbb{T})\ar@{=>}[ul]}
$$

\noindent Note that statement (a) follows from the fact that Matet forcing
preserves p-points which was proved by \cite[Theorem 4]{Eisworth2002-EISFAS}. Our proof
is more direct and combinatorial. 

\section{Minimal degree and avoiding quasi-generics}
Matet forcing consists of pairs $(s, A)$, where $s \in \omega^{< \omega}$ is strictly increasing and $A \subseteq [\omega]^{< \omega}$ is infinite, and for all $a \in A$, $max(ran(s)) < min(ran(a))$. $(t, B) \leq (s, A) \text{ iff }$: 

\[s\subseteq t \forall b \in B \exists A' \subseteq A (|A'| < \omega \wedge b = \cup A') \wedge \mathrm{ran}(t) \setminus \mathrm{ran}(s) \subseteq A\]

A moment of reflection shows that the Matet conditions can be seen as Miller trees, but with a different ordering.

At the heart of almost all the arguments that shall be illustrated here is the fusion technique. But before defining that we shall need to define some terminology.
$A \sqsubseteq B$ (read $A$ is a condensation of $B$) iff every $a \in A$ is a finite union of elements of $B$. For any finite subset $t$ of $\omega$, $A \mathrm{past} t$, denotes the subset of $A$, consisting of all the $a$, such that $\mathrm{min}(a) > \mathrm{max}(t)$. We shall by abuse of notation say that $t < a$, when $\mathrm{min}(a) > \mathrm{max}(t)$. 

If $(s, A)$ is a Matet condition and $A_{n}$ is a sequence such that $A_{n+1} \sqsubseteq A_{n} \mathrm{past} a^{0}_{n}$, where $a^{0}_{n}$ is the first element of $A_{n}$ with respect to $<$, and if $B \ = \ \{a^{0}_{n} : n \in \omega\}$, then $(s,B) \leq (s, A)$. We shall denote by $\mathrm{FU}(A)$ the set of all finite unions of elements of $A$. 

Our initial objective is to show that Matet forcing has \textit{pure decision property} and \textit{adds reals of minimal degree}. That is if $V$ is the ground model and $x$, a real added by Matet forcing, then $x$ can recover the generic real added by the Matet forcing.

The proof of the first can be found in \cite[Lemma 2.6]{Eisworth2002-EISFAS} but we include it here as it is an integral part of the argument. The proofs of the above two shall also illustrate the fusion technique in detail.
\begin{thm}\label{pure dec}
Let $\theta$ be a sentence and $(s,A)$ a Matet condition. Then there is an extension $(s,B)$ such that for any $t \in \mathrm{FU}(B)$, $(s,B \mathrm{past} t)$ decides $\theta$.
\end{thm}

\begin{proof}
 We shall by induction define a sequence $A_{n}$ starting with $A_{-1} = A$, such that:
 \begin{itemize}
     \item $A_{n+1} \sqsubseteq A_{n} \mathrm{past} a^{0}_{n}$.
     \item $(s \cup a^{0}_{n}, A_{n}\mathrm{past} a^{0}_{n})$ decides $\theta$.
 \end{itemize}

 Given, $A_{n}$, we can simply find an extension of $(s \cup a^{1}_{n}, A_{n}\mathrm{past}a^{1}_{n})$ say $(s\cup t, A_{n}')$ such that it decides $\theta$. We let $A_{n+1} = A_{n}' \cup \{t\}$.

 We now set $A' = \{a^{0}_{n} \ : \ n \in \omega\}$. Now, either for infinitely many of $t \in A'$, $(s \cup t, A'\mathrm{past}t) \Vdash \theta$ or for infinitely many of them $(s \cup t, A'\mathrm{past} t) \Vdash \neg \theta$. We set $B = \{t \in A' \ : \ (s \cup t, A'\mathrm{past} t)\Vdash \theta\}$ or $B = \{t \in A' \ : \ (s \cup t, A'\mathrm{past} t)\Vdash \neg \theta\}$ depending on which is infinite.

 $(s,B)$ is therefore the required extension.
\end{proof}
To the extent of proving the second one, let $\dot{x}$ be a name for a non-ground model real and $(s,A)$ be a Matet condition forcing that.

By virtue of pure decision property (\ref{pure dec}), we can assume that for every $t \in \mathrm{FU}(A)$, there is a ground model real called the guiding real $x_{t}$, such that \[(s \cup t, A\mathrm{past} t_{k}) \Vdash \dot{x}\restriction k = x_{t}\restriction k\] for all $k \in \omega$.\\

The above can be proven by slightly modifying the proof of \cite[Lemma 2.7]{Eisworth2002-EISFAS}. We shall however sketch it here for the sake of completeness. 
\begin{proof}
    First of all we notice that, it is possible to inductively define a sequence $A_{n}$, starting with $A_{-1} = A$, such that $A_{n+1} \sqsubseteq A_{n}\mathrm{past}a^{0}_{n}$, and $(s,A_{n})$ decides $\dot{x}$ upto $n$. Now, considering $(s,B)$, where $B = \{a^{0}_{n} \ : \ n \in \omega\}$, we have that there is a guiding real corresponding to $s$, say $x_{s}$. 
    
    Now, we shall use the above argument repetitively in an inductive manner to arrive at the required condition. We need to define a sequence $B_{n}$, starting with $B_{-1} = B$ such that:
    \begin{itemize}
        \item $B_{n+1} \sqsubseteq B_{n} \mathrm{past} b^{0}_{n}$.
        \item for all $t \in \mathrm{FU}(\{b^{0}_{k} \ : \ k \leq n\})$, we have $(s \cup t, B_{n+1} \mathrm{past} t)$ satisfying the condition that there is a guiding real corresponding to $t$, say $x_{t}$.
    \end{itemize}

    Given $B_{n}$, we simply enumerate the terminal nodes of $\mathrm{FU}(\{b^{0}_{k} \ : \ k \leq n\})$ as $t_{0}, ... , t_{m}$. Set $C_{0} = B_{n}\mathrm{past}b^{0}_{n}$. Given $C_{i}$, set $C_{i+1} \sqsubseteq C_{i}\mathrm{past}t_{i+1}$, such that for $(s \cup t_{i + 1}, C_{i+1})$, there is a guiding real $x_{t_{i+1}}$. We set $B_{n+1} = C_{m}$. Setting $C = \{b^{0}_{n} \ : \ n \in \omega\}$, we have $(s,C)$ to be the required condition.
\end{proof}
\begin{thm}\label{minimal degree}
    Matet forcing adds reals of minimal degree
\end{thm}
\begin{proof}
  We are going to build a sequence $A_{n}$, starting with $A_{-1} = A$ and for every $t \in \mathrm{FU}(\{a^{0}_{j} : j \leq n\})$, $\dot{x}_{s\cup t}$ denotes the maximal initial segment decided by $(s\cup t, (A_{n+1}\cup\{a^{0}_{j} : j \leq n+1\}) \mathrm{past} t)$ and $\ell(s \cup t)$, denotes the largest set according to $<$ in $A_{n+1}\cup\{a^{0}_{j} : j \leq n+1\}$ which is a subset of $s \cup t$.

  Now, that we have set up the terminology, we can proceed with the proof. We require that for every $n \in \omega$,  $A_{n+1} \sqsubseteq A_{n} \mathrm{past} a^{0}_{n}$ and for every $t \in \mathrm{FU}(\{a^{0}_{j} : j \leq n\})$,
\[(s\cup t, (A_{n+1}\cup\{a^{0}_{j} : j \leq n+1\}) \mathrm{past} t)\Vdash \dot{x}_{s \cup t} \neq x_{s \cup t \restriction \mathrm{min}(\ell(s \cup t))}\restriction |\dot{x}_{s\cup t}|\]
and
\[(s\cup t \restriction \mathrm{min}(\ell(s \cup t)), (A_{n+1}\cup\{a^{0}_{j} : j \leq n+1\}) \mathrm{past} t)\Vdash \dot{x}_{s \cup t\restriction \mathrm{min}(\ell(s \cup t))} = x_{s \cup t\restriction \mathrm{min}(\ell(s \cup t))} \restriction |\dot{x}_{s\cup t}|\]

$A_{n+1}$ is actually constructed inductively. Let's say we enumerate $\mathrm{FU}\{a^{0}_{j} \ : \ j \leq n\}$ as $(t_{i})_{i \in m}$, we set $B_{0} = A_{n} \mathrm{past} a^{0}_{n}$. We form a sequence $(B_{i})_{i \in m}$, such that $B_{i+1} \sqsubseteq B_{i}$, $b^{0}_{i} \subseteq b^{0}_{i+1}$, and for all $i \in m$, \[(s \cup t_{i} \cup b^{0}_{i}, B_{i}\mathrm{past}b^{0}_{i})\Vdash \dot{x}_{s \cup t_{i} \cup b^{0}_{i}} \neq x_{s \cup t_{i} }\restriction |\dot{x}_{s\cup t_{i} \cup b^{0}_{i}}|\]
and
\[(s \cup t_{i}, B_{i}\mathrm{past}b^{0}_{i})\Vdash \dot{x}_{s \cup t_{i}} = x_{s \cup t_{i} }\restriction |\dot{x}_{s\cup t_{i} \cup b^{0}_{i}}|\]

This is possible because all the guiding reals are ground model and $(s, A)$ forces that $\dot{x}$ is not ground model. Finally, $A_{n+1} = B_{m}$.

Now, we just let $B = \{a^{0}_{n} : n \in \omega\}$. Then, we have that the function $f : [(s, B)] \rightarrow 2^{\omega}$ defined as \[f(x) = \bigcup_{k \in \omega} \dot{x}_{s \cup \bigcup_{k \in \omega}b^{n_{k}}}\] where, $x = s \cup \bigcup_{k \in \omega}b^{n_{k}}$, to be a continous injective ground model one, such that $(s,B) \Vdash f(x_{G}) = \dot{x}$. \end{proof}

Note that this proves the fact that Matet forcing adds reals of minimal degree. Moreover the function $f$ is ground model and continuous.

Now, we shall come to the part of avoiding quasi-generics of closed locally countable graphs. That is if $G$ is a closed locally countable graph, we shall show that for any real say $r$ added by the Matet forcing, it is contained in a ground model Borel set $B$, such that $B$ is $G$ independent, that is any two elements of $B$ do not form a $G$ edge.  

A graph $G$ on $2^{\omega}$ is said to be closed locally countable, iff $G$ is a closed subset of $2^{\omega} \setminus \Delta$, where $\Delta$ is the diagonal (in other words the equality realtion on $2^{\omega}$ is the $\Delta$) and for every $x$, there are atmost countably many $y$ such that $(x,y) \in G$.

As said earlier it will also boil down to something that is ultimately a fusion argument. If we notice carefully, it is easy to observe that\[T_{(s,B)}(\dot{x}) = \{\dot{x}_{(t,C)} \ : \ (t,C) \leq (s,B)\}\] is a perfect tree.
\begin{thm}
    Matet forcing does not add quasi-generics of closed locally countable graphs
\end{thm}
\begin{proof}
  Now, we look forward to create once again a sequence $A_{n}$ as before, starting with $A_{-1} = A$ such that $A_{n+1} \sqsubseteq A_{n} \mathrm{past} a^{0}_{n}$ and for every $t \in \mathrm{FU}(\{a^{0}_{j} : j \leq n\})$ we have 

\[(s\cup t, (A_{n+1}\cup\{a^{0}_{j} : j \leq n+1\}) \mathrm{past} t)\Vdash ([\dot{x}_{s \cup t}] \times [x_{s \cup t \restriction \mathrm{min}(\ell(s \cup t))}\restriction |\dot{x}_{s\cup t}|]) \cap G = \emptyset\]

Like in the proofs of \ref{minimal degree} and \ref{pure dec}, given $A_{n}$, we enumerate $\mathrm{FU}(\{a^{0}_{j} \ : \ j \leq n\})$ as $(t_{i})_{i \in m}$ and set $B_{0} = A_{n} \mathrm{past} a^{0}_{n}$ and we define a sequence $(B_{i})_{i \in m}$ such that, $B_{i+1} \sqsubseteq B_{i}$, $b^{0}_{i} \subseteq b^{0}_{i+1}$, and \[(s \cup t_{i} \cup b^{0}_{i}, B_{i}\mathrm{past}b^{0}_{i})\Vdash ([\dot{x}_{s \cup t_{i} \cup b^{0}_{i}}] \times [x_{s \cup t_{i} }\restriction |\dot{x}_{s\cup t_{i} \cup b^{0}_{i}}|]) \cap G = \emptyset\]

This is possible, due to the fact that $T_{(s,B)}$ is a perfect tree and choosing $b^{0}_{i+1}$, long enough, we shall have $(x_{s \cup t_{i} \cup b^{0}_{i+1}}, x_{s \cup t_{i}}) \notin G$, and $B_{i+1}$ is then obtained by deleting sufficiently many elements of $B_{i}\mathrm{past}b^{0}_{i+1}$ in an increasing order, since for sufficiently long initial segments $\sigma$ and $\tau$ of $x_{s \cup t_{i} \cup b^{0}_{i+1}}$ and $x_{s \cup t_{i}}$ respectively, we have $([\sigma]\times[\tau]) \cap G = \emptyset$, due to closedness of $G$. $B_{m} = A_{n+1}$ and we define $C$ to be $\{a^{0}_{n} \ : \ n \in \omega\}$. Then, $(s,C)$ is the required condition for which $[T_{(s,C)}(\dot{x})]$ is $G$ independent that is for any two elements $x$ and $y$ of it, $(x,y) \notin G$. It is also a ground model closed set. Moreover $(s,C) \Vdash \dot{x} \in [T_{(s,C)}(\dot{x})]$. This completes the proof.  
\end{proof}

\section{Iteration of Matet forcing}

The argument of the single case can be generalised to the iteration case but for that one needs to first come up with some new terminology.

Let $\alpha$ be an ordinal such that $\alpha$ $<$ $\omega_{2}$. Then, if $(s,A) \in \mathbb{T}_{\alpha}$, and $F \subseteq \mathrm{supp}(s,A)$, finite and $k \in \omega$. We say that $(t,B) \leq_{F,k} (s,A)$ iff for all $\gamma \in F$ $(t,B)\restriction \gamma \Vdash t(\gamma) = s(\gamma)$ and $(b_{i}(\gamma) = a_{i}(\gamma))$ for all $i \leq k(\gamma)$.  

A fusion sequence consists of the following:

\begin{itemize}
    \item $F_{n},k_{n}$, where $k_{n} : F_{n} \rightarrow \omega$
    \item $\bigcup_{n \in \omega} F_{n} = \mathrm{supp}(s,A)$
    \item $k_{n+1}(\gamma) \geq k_{n}(\gamma)$, for all $\gamma \in F_{n}$ 
    \item $(s_{n},A_{n})$ such that $(s_{n+1},A_{n+1}) \leq_{F_{n},k_{n}} (s_{n},A_{n})$
    
\end{itemize}

Our attempt is to build a fusion sequence $(s_{n},A_{n}), F_{n}, k_{n}$, such that if $(s,A)$ is the fusion of $(s_{n},A_{n})$, then $T_{(s,A)}$ is $G$ independent.

For the sake of notational convenience we shall be identifying the nodes of a Matet tree with $\omega^{< \omega}$ with the help of the natural order preserving-bijection such that $(s,A)\ast\emptyset = \mathrm{st}(s,A)$ and $(s,A)\ast(\sigma^{\smallfrown}n)$ is the $nth$ immediate successor of $(s,A)\ast\sigma$ according to the lexicographic ordering on $T_{(s,A)}$. We shall be using the notations $(s,A)\ast\sigma$ and $T_{(s,A)}\ast\sigma$ interchangeably.

We say that a $\mathbb{T}_{\alpha}$ is a $(F_{n}, k_{n})$ faithful condition iff for any two pairs $\sigma$ and $\sigma'$ $\in$ $\Pi_{\gamma \in F_{n}}k(\gamma)^{k(\gamma)}$, such that $\sigma \neq \sigma'$, $([T_{(s,A)}\ast\sigma] \times [T_{(s,A)}\ast\sigma']) \cap G = \emptyset$.

\begin{lem}
Suppose that $(s,A)$ is $(F_{n},k_{n})$ faithful and $k'_{n}$ is such that $k'_{n}(\gamma) = k_{n}(\gamma) + 1$ and for all $\beta \in F_{n}\setminus \{\gamma\}$, $k'_{n}(\beta) = k_{n}(\beta)$. Then one can find $(s,B) \leq_{F_{n}, k_{n}} (s,A)$, such that $(s,B)$ is $(F_{n},k'_{n})$ faithful.
\end{lem}
\begin{proof}
    Let us denote $\mathrm{max}(F_{n})$ as $\gamma_{\mathrm{max}}$. Let, $\{\sigma_{0},\sigma_{1}, ... , \sigma_{k}\}$ be an enumeration of $\Pi_{\gamma \in F_{n}} k(\gamma)^{k(\gamma)}$. We inductively define a $\leq_{F_{n},k_{n}}$ decreasing sequence $(s,B_{m})$, such that:

    \begin{itemize}
    \item  $([T_{(s,B_{m})\ast\sigma_{m}(\gamma_{\mathrm{max}})^{\smallfrown}n}(\dot{x})]\times[T_{(s,B_{m})\ast\sigma_{m}(\gamma_{\mathrm{max}})^{\smallfrown}n'}(\dot{x})]) \cap G = \emptyset$ 
    
\end{itemize}
   
    Suppose that $(s,B_{m-1})$ has already been defined. Then, due the fact that $G$ is closed and locally countable, just as in the single-step, $(s,B_{m-1})\ast\sigma_{m}\restriction\gamma_{\mathrm{max}}$ forces that for every $n \in \omega$, there is a tail $t_{n}$ $ \leq (s,B_{m-1})\ast\sigma_{m}(\gamma_{\mathrm{max}})^{\smallfrown}n[\gamma_{\mathrm{max}},\alpha)$, such that for $n \neq n'$:

    \[([T_{t_{n}}(\dot{x})]\times[T_{t_{n'}}(\dot{x})]) \cap G = \emptyset\]

    Therefore, one can find $(s,B_{m}) \leq_{(F_{n}, k_{n})} (s, B_{m-1})$, such that $(s,B_{m})\ast\sigma_{m}\restriction \gamma_{\mathrm{max}}$ forces that for all $n \in \omega$ there exists some $p_{n} \in \omega$ 
 such that $(s,B_{m})\ast\sigma_{m}(\gamma)^{\smallfrown}n = t_{p_{n}}$ and therefore we have that: 
    \[([T_{(s,B_{m})\ast\sigma_{m}(\gamma_{\mathrm{max}})^{\smallfrown}n}(\dot{x})]\times[T_{(s,B_{m})\ast\sigma_{m}(\gamma_{\mathrm{max}})^{\smallfrown}n'}(\dot{x})]) \cap G = \emptyset\]

    We let $(s,B)$ to be $(s,B_{k})$. This completes the proof.
    
\end{proof}

\begin{thm}\label{Matet vs. Silver}
    $V^{\mathbb{T}_{\omega_{2}}}$ has no quasi-generics of closed locally countable graphs.
\end{thm}
\begin{proof}
    Using the previous lemma, for every $\alpha \in \omega_{2}$, and $(s,A) \in \mathbb{T}_{\alpha}$, one can construct a fusion sequence as $(s_{n},A_{n})$ as above and define the fusion of it as $(s,B)$ such that \[\forall \gamma \in \alpha ((s,B)\restriction\gamma\Vdash \forall n \in \omega (B \sqsubseteq A_{n}))\].

    We now define a function $f : (\omega^{\omega})^{\mathrm{supp}(s,B)} \rightarrow 2^{\omega}$, such that if $x(\gamma)_{\gamma \in \mathrm{supp}(s,B)}$ is mapped to:
    \[\bigcup_{n \in \omega} \dot{x}_{(s,B)\ast(x(\gamma)\restriction k_{n}(\gamma))_{\gamma \in F_{n}}}\]

    Notice that this is a ground model Borel injective map and it maps the generic to $\dot{x}$.
    
\end{proof}

\begin{cor}
$\chi_{\mathrm{B}}(G)$ where $G$ is a locally countable graph is small in $V^{\mathbb{T}_{\omega_{2}}}$
\end{cor}

\begin{cor}\label{cor:matetvssilver}
$\boldsymbol{\Delta}^{1}_{2}(\mathbb{V})$ doesn't hold in $V^{\mathbb{T}_{\omega_{2}}}$.
\end{cor}
\begin{proof}
    The reason being that $G_{1}$ is a closed locally countable graph.
\end{proof}

\section{Separating Willow and Matet regularities}
In this section we are going to focus on separation of regularity properties. In particular we are going to show that in the model generated by the $\omega_{2}$ iteration of Willow tree forcing $\boldsymbol{\Delta}^1_2(\mathbb{T})$ does not hold.

\begin{thm}
    $\mathbb{W}$ is $\omega^{\omega}$ bounding.
\end{thm}

\begin{proof}
    Suppose that we have a willow tree $T$, and $\dot{x}$ is a name for a real in $\omega^{\omega}$. Then, if $\sigma$, is the first splitting node of $T$, there is an extension $q'_{0} \leq T_{\sigma^{\smallfrown}0}$, such that $q'_{0}$, decides $\dot{x}(0)$. Now, let $q'_{1} \leq T_{\sigma^{\smallfrown }1}$, such that $\mathrm{dom}(q'_{0}) = \mathrm{dom}(q'_{1})$ and $(q'_{0})^{-1}\{1\} = (q'_{1})^{-1}\{1\}$. We find an extension $q_{1}$ of $(q_{1})'$, such that it decides $\dot{x}(0)$ too. Let $q_{0}$, be an extension of $q'_{0}$, such that $\mathrm{dom}(q_{0}) = \mathrm{dom}(q_{1})$ and $(q_{0})^{-1}\{1\} = (q_{1})^{-1}\{1\}$. Then, clearly $T_{0} = q_{0} \cup q_{1}$, is a Willow tree. Let, $m_{0} = \mathrm{max}\{\dot{x}_{q_{0}}(0), \dot{x}_{q_{1}}(0)\}$.

    Suppossing, that we have already found $T_{i-1} \leq_{i-2} T_{i-1} \leq_{i-3} .... \leq_{0} T_{0}$, where $T_{0} = T$ and also we have found $(m_{j})_{j \leq i}$.

    Then, we define $T_{i}$ by induction on the lexicographic ordering of $2^{i}$. Suppose that $\sigma_{0}$ is the least element in this ordering. We repeat the procedure described in the above paragraph to obtain $(T_{i-1}\ast\sigma_{0})_{0}$, except for the fact that we demand this time that it decides $\dot{x}(i)$. Now, we copy the tree above the stem of $(T_{i-1}\ast\sigma_{0})_{0}$ to all the $T_{i-1}\ast\sigma_{k}'s$. This gives us $T_{i}^{0}$.

    Supposing that we have already obtained $T_{i}^{m}$, we can determine $T_{i}^{m+1}$, we proceed to determine $(T^{m}_{i}\ast\sigma_{m+1})_{0}$, such that it decides $\dot{x}(i)$, and again copy over the tree above the stem of $(T^{m}_{i}\ast\sigma_{m+1})_{0}$ to all the $T_{i}^{m}\ast\sigma_{k}'s$.

    $T_{i}^{2^{i}}$, is the required $T_{i}$.

    The fusion of all the $T_{i}'s$, that is $\bigcap_{i \in \omega} T_{i}$, is our required condition say $S$, such that $S \Vdash \dot{x} <* g$, where $g(i) = m_{i}$ and $g$ is ground model.
\end{proof}

\begin{thm}\label{unbounded}
    $\boldsymbol{\Delta}^1_2(\mathbb{T}) \implies$ that there is an unbounded real over L[a] for every $a \in \omega^{\omega}$. 
\end{thm}
\begin{proof}
    A simple observation shows that Matet trees are also superperfect trees (Miller trees). Therefore one can directly see that the proof of \cite[6.1]{brendlelwe} works in this case because it relies on the fact that the trees are superperfect and since Matet trees are superperfect, we are done. (should I sketch the proof here?)
\end{proof}
\begin{cor}\label{Matet vs. Silver and Willow}
    $\boldsymbol{\Delta}^1_2(\mathbb{T})$ does not hold $V^{\mathbb{W}_{\omega_{2}}}$ and $V^{\mathbb{V}_{\omega_{2}}}$.
\end{cor}
\begin{proof}
   $\omega^{\omega}$ bounding is preserved by iteration. 
\end{proof}

\section{Willow and Matet regularities in the Sacks Model}

\begin{thm}\label{Sacks}
    $\boldsymbol{\Delta}^1_2(\mathbb{T})$ does not hold in the Sacks model.
\end{thm}

\begin{proof}
    Follows from the fact that Sacks forcing doesn't add unbounded reals and theorem \ref{unbounded}
\end{proof}

\begin{thm}\label{sacksmain}
    Countable support iteration of length $\omega_{1}$ of Sacks forcing does not add Willow quasi-generics
\end{thm}

The trick lies in ensuring that for every condition $p$ $\in$ $\mathbb{S}_{\alpha}$ where $\alpha < \omega_{1}$, one can find an extension $r$ of $p$, such that $T_{r}(\dot{x})$ has all it's splitting levels at different heights. The fact that $T_{r}(\dot{x})$ is Borel will ensure that it is Willow regular, but at the same time for any willow tree $T$, $[T] \nsubseteq T_{r}(\dot{x})$.

We shall here be assuming that there is a groundmodel homeomorphism $h : (2^{\omega})^{supt(q)} \rightarrow T_{p}(\dot{x})$ as outlined in \cite[Lemma 78]{Sacks}.

Now for proving the theorem we shall first of all develop the required terminology:

Given a finite set $F \subseteq \mathrm{supt}(p)$ and $\eta : F \rightarrow \omega$, we say that a condition $q \leq p$ is $(F,\eta)-faithful$ if for any two elements $\sigma$ and $\tau$ of $\Pi_{\gamma \in F} 2^{\eta (\gamma)}$, $|\dot{x}_{q\ast\sigma}| \neq |\dot{x}_{q\ast\tau}|$. 

Also for any two conditions $q$ and $p$ in $\mathbb{S}_{\alpha}$, we say that $q \leq_{(F,\eta)} p$, if for all $\sigma \in \Pi_{\gamma \in F} 2^{\eta (\gamma)}$, $q\ast\sigma$ $=$ $p\ast\sigma$. 

Our goal is to build a sequence $(p_{n}, F_{n}, \eta_{n})$ which satisfies the following properties:

\begin{itemize}
    
    \item $p_{n+1} \leq_{(F_{n},\eta_{n})} p_{n}$
    \item $p_{n}$ is $(F_{n}, \eta_{n})-faithful$
    \item $F_{n} \subseteq F_{n+1}$
    \item $\bigcup_{n \in \omega} F_{n}$ $=$ $\mathrm{supt}(p)$.
    \item $\eta_{n}(m) \leq \eta_{n+1}(m)$ for all $m \in F_{n}$
\end{itemize}

To this end the following lemma plays a crucial role.

\begin{lem}
    Suppose, that $\alpha < \omega_{1}$ is an ordinal, $p$ an $\mathbb{S}_{\alpha}$ condition, $F \subseteq \alpha$ is finite, $\eta : F \rightarrow \omega$, $\eta' : F \rightarrow \omega$ are such that $\eta \restriction F\setminus\{\beta\} = \eta' \restriction F\setminus\{\beta\}$ and $\eta'(\beta) = \eta(\beta) +1$. Moreover let $p$ be $(F,\eta)-faithful$. Then, there exists $q \leq_{(F, \eta)} p$ such that for all $\sigma, \tau \in \Pi_{\gamma \in F}2^{\eta'(\gamma)}$, $|\dot{x}_{q\ast\sigma}| \neq |\dot{x}_{q\ast\tau}|$.
\end{lem}

\begin{proof}
    Suppose we have an enumeration $\{\sigma_{1},..., \sigma_{m}\}$ of $\Pi_{\gamma \in F}2^{\eta(\gamma)}$. Then we shall inductively build a $\leq_{(F, \eta)}$ decreasing sequence $q_{i}$.

    Suppose that we have already found $q_{i-1}$. Then, we take $q_{\sigma_{i},0}$ and $q_{\sigma_{i},1}$ to be such that  and $q_{i-1} \ast \sigma_{i}\restriction \delta$ forces the following:
\begin{itemize}
    \item $q_{\sigma_{i},k} \leq q(\delta)\ast q(\sigma_{i}^{\smallfrown}k)^{\smallfrown} q\restriction (\delta, \alpha)$
    \item $|\dot{x}_{q_{\sigma_{i},k}}| > |\dot{x}_{q_{\sigma_{i}}}|$
    \item $|\dot{x}_{q_{\sigma_{i},0}}| < |\dot{x}_{q_{\sigma_{i},1}}|$
\end{itemize}

One can now choose a condition $q_{j} \leq_{(F,\eta)} q_{j-1}$ such that $q_{j}\ast\sigma_{j}\restriction\delta \Vdash q_{j}(\delta)\ast\sigma_{j}(\delta)^{\smallfrown}k^{\smallfrown}q_{j}^{\smallfrown}q(\delta,\alpha) = q_{\sigma_{i},k}$. 

Then our required $q$ is simply $q_{m}$.
\end{proof}

 Using the above lemma one can construct a fusion sequence $(p_{n}, F_{n}, \eta_{n})$, such that it's fusion say $r$, is such that $T_{r}(\dot{x})$ is a tree with splitting levels all at different heights. This completes the proof of \ref{sacksmain}.

 \begin{cor}\label{Sacks vs Matet}
     $\boldsymbol{\Delta}^1_2(\mathbb{T})$ does not hold in $V^{\mathbb{S}_{\omega_{2}}}$.
 \end{cor}

 \section{Willow and Matet regularities in the Miller model}

 The idea is to show that \ref{Sacks} holds when Sacks is replaced by Miller.

 \section{$E_{0}$ and Willow}
It is clear that $\boldsymbol{\Delta}^1_2(\mathbb{W})$ implies $\boldsymbol{\Delta}^1_2(E_0)$. In this section we are going to prove that the converse doesn't hold.

The fact that $E_{0}$ forcing adds reals of minimal degree is in fact a direct adaptation of the proof that $V$ adds reals of minimal degree. But an observation of this proof will lead to the following corollary:

\begin{cor}\label{E_0 corr}
    For every $E_0$ tree $p$ there exists $q \leq p$ such that every splitting node $s$ of $q$ is such that $\dot{x}_{q_{s^{\smallfrown}0}} \neq \dot{x}_{q_{s^{\smallfrown}1}}$.
\end{cor}

Now the above corollary just points out that every splitting node in $q$ corresponds to a splitting in $T_{q}(\dot{x})$.

We shall from now on assume that $p$ itself satisfies \ref{E_0 corr}.
\begin{thm}
    Single step $E_{0}$ forcing doesn't add $\mathbb{W}$ quasi-generics.
\end{thm}

\begin{proof}
  Now, we shall have the following two scenarios:

\begin{itemize}
    \item $\exists q \leq p$ $\forall r \leq q$ and $\forall s,t \in r$
    \[|\dot{x}_{r_{s}}| \neq |\dot{x}_{r_{t}}|\]
    \item $\forall q \leq p$ $\exists r \leq q$ and $\exists s,t \in r$
    \[|\dot{x}_{r_{s}}| = |\dot{x}_{r_{t}}|\]
\end{itemize}

In the first case we just have $T_{q}(\dot{x})$ is very asymmetrical and therefore a $\mathbb{W}$ small set.

In the second case let $s_{0}$ and $s_{1}$ be the minimal splitting nodes above $\mathrm{st}(r)$ such that $|\dot{x}_{s_{0}}| = |
\dot{x}_{s_{1}}|$.

Then, let $i$ be such that there is $r^{i} \leq r_{s_{0}^{\smallfrown}i}$ $1-\dot{x}_{s_{0}}(|\dot{x}_{\mathrm{st}(r)}|) = \dot{x}_{r^{i}}$.

Similarly let $j$ be such that the above procedure can be done for $r_{s_{1}}$ and also let $r^{j} \leq r_{s_{1}}$ be the corresponding extension.

We find then $r'_{0} \leq r^{i}$ and $r'_{1} \leq r^{j}$ be such that $\mathrm{st}(r^{i}) = \mathrm{st}(r^{j})$.

We let $r_{0}$ to be $r'_{0} \cup r'_{1}$. Now if $r_{n-1} \leq_{n-1} r_{n-2} \leq_{n-2} ... r_{0} \leq_{0} r$ have already been defined, then it is easy to notice that one can repeat the above procedure for each of the nodes of the $n^{th}$ splitting level of $r_{n-1}$ and copying over to get $r_{n} \leq_{n} r_{n-1}$. 

The fusion of this sequence say $q$ is easily seen to be such that $T_{q}(\dot{x})$ is $\mathbb{W}$ independent.  
\end{proof}

\bibliographystyle{plain}
\bibliography{bibliography}

\end{document}